\documentclass[12pt]{amsart}
\usepackage{amssymb}
\usepackage{verbatim}
\usepackage{mathrsfs}
\usepackage{mathtools}

\usepackage{xcolor}

\newtheorem{thm}{Theorem}[section]
\newtheorem{lem}[thm]{Lemma}
\newtheorem{rema}[thm]{Remark}

\numberwithin{equation}{section}

\newcommand{\Zag}{\mathscr{L}}
\newcommand{\M}{\mathcal{M}}
\newcommand{\al}{\alpha}
\newcommand{\DS}{\mathcal{D}}
\newcommand{\DSV}{\mathcal{D}}
\providecommand{\arcosh}{\operatorname{arcosh}}
\newcommand{\TM}{\vartheta}
\newcommand{\Res}{\mathcal{R}}
\newcommand{\STM}{\Theta}
\newcommand{\W}{\mathcal{W}}

\providecommand{\sym}{\operatorname{sym}}
\newcommand{\Mod}[1]{\ (\textup{mod}\ #1)}
\providecommand{\sgn}{\operatorname{sgn}}

\mathtoolsset{showonlyrefs}

\begin{document}

\title[The second moment of Maass form symmetric-square $L$-functions]{The second moment of Maass form symmetric square $L$-functions at the central point}

\author[D. Frolenkov]{Dmitry  Frolenkov}
\address{
Steklov Mathematical Institute of Russian Academy of Sciences, 8 Gubkina st., Moscow, 119991, Russia}
\email{frolenkov@mi-ras.ru}

\begin{abstract}
Recently R.Khan  and M.Young proved a mean Lindel\"{o}f estimate on the second moment of central values of Maass form symmetric-square $L$-function on the interval $T<|t_j|<T+T^{1/5+\epsilon}$, where $t_j$ is a spectral parameter of the Maass form. We provide another proof of this result.
\end{abstract}


\maketitle


\section{Introduction}
Denote by $u_j$ a Hecke-Maass cusp form with Laplace eigenvalue $\kappa_{j}=1/4+t_{j}^2$.  Khan and Young proved  the following result (see \cite[Theorem 1.3]{KhYoung}).
\begin{thm}\label{KY result}
For $G>T^{1/5+\epsilon}$ and $0\le U\ll T^{\epsilon}$ one has
\begin{equation}\label{mean Lindelof}
\sum_{T<t_j\le T+G}\alpha_{j}|L(\sym^2 u_{j},1/2+iU)|^2\ll T^{1+\epsilon}G.
\end{equation}
\end{thm}
Previously \eqref{mean Lindelof} was established for $G>T^{1/3+\epsilon}$ by Lam \cite{Lam} (who proved a similar result in the case of holomorphic cusp forms) and by Tang-Xu \cite{TangXu}.  Our goal is to provide another proof of Theorem \ref{KY result}. To compare the two methods, we first discuss briefly the original proof by  Khan and Young of Theorem \ref{KY result} (see \cite[sec.1.3]{KhYoung} for the detailed overview).
As the first step, Khan and Young wrote down an approximate functional equation for each of the two L-functions $L(\sym^2 u_{j},\cdot)$, obtaining an expression for which the Kuznetsov trace formula can be used. Application of the trace formula, in turn, yields sums of Kloosterman sums of the following shape:
\begin{equation}
\sum_{m,n,c}\frac{S(m^2,n^2,c)}{m^{1/2+iU}n^{1/2-iU}c}H\left(\frac{4\pi mn}{c}\right).
\end{equation}
As the next step, Khan and Young evaluated the integral $H(x)$ which was derived from the integral transforms in the Kuznetsov formula.  After that they split the summation over variables $m$ and $n$ into arithmetic progressions modulo $c$ and applied the Poisson summation on both $m$ and $n$, reducing the problem to the analysis of the triple sum:
\begin{equation}
\sum_{k,l,c}\frac{1}{c^2}T(k,l,c)I(k,l,c),
\end{equation}
where $T(k,l,c)$ is some complicated exponential sum related to Gauss sums (see \cite[(5.14)]{KhYoung}), and $I(k,l,c)$ is a double integral.  Note that the sum $T(k,l,c)$ can be evaluated explicitly as was shown in \cite[Lemma 5.3]{KhYoung}, however the answer is quite complicated and it is required to consider separately several different cases. Using the stationary phase method, it is also possible to prove an asymptotic formula for
$I(k,l,c)$. Further, Khan and Young separated the variables $k$, $l$ and $c$ via the Mellin inversion, obtaining a multiple integral and a multiple Dirichlet series.
A delicate analysis of this multiple Dirichlet series allows rearranging it as a series of Dirichlet $L$-functions of quadratic characters. Applying the functional equation to the Dirichlet $L$-function (which can be viewed as the third application of the Poisson summation formula) and continuing to operate with multiple series, Khan and Young finally obtained the second moment of Dirichlet $L$-functions.  This moment was estimated with the help of Heath-Brown's large sieve inequality for quadratic characters completing the proof of Theorem \ref{KY result}.

Now we describe main steps of our proof of Theorem \ref{KY result}. For the sake of simplicity, we consider only the case $U=0$, though the method also works for all $U\ll T^{\epsilon}$. Instead of applying the approximate functional equation twice, we use it only for one of the L-functions $L(\sym^2 u_{j},1/2)$. In doing so, we reduce the problem to the study of the first  moment of symmetric-square $L$-functions twisted by the Fourier coefficient $\lambda_j(m^2)$. For this moment, we apply a reciprocity type formula relating it to the sum $$\sum_{n}\Zag_{n^2-4m^2}(1/2)\frac{I(n/m)}{\sqrt{n}}$$ of the so-called Zagier L-series which can be considered as a certain generalization of the Dirichlet $L$-function of quadratic characters. Here $I(x)$ is some integral of a hypergeometric function. Note that the proof of this reciprocity formula contains application of the Kuznetsov trace formula and one application of the Poisson summation. Finally, it is required to study the double sum:
\begin{equation}
\sum_{n,m}\Zag_{n^2-4m^2}(1/2)\frac{I(n/m)}{\sqrt{nm}}.
\end{equation}
A straightforward attempt to sum over $n$ or over $m$ seems not to lead to a good answer. Therefore, we make the following change of variables:
\begin{equation}
n-2m=q,\quad n+2m=r.
\end{equation}
It turns out that $I(n/m)$ is negligible unless $q\ll T/G^2$ and $qG^2\ll r\ll T$. Making one more change of variable $r=l/q$ and representing the congruence condition $l\equiv0\pmod{q}$ with the use of additive harmonics, we obtain
\begin{equation}
\sum_{q\ll T/G^2}\sum_{c|q}\sum_{\substack{0\le a<c\\(a,c)=1}}\sum_{l\ll qT}\Zag_{l}(1/2)e\left(\frac{al}{q}\right)F(q,l).
\end{equation}
Now we are ready to apply a suitable Voronoi summation formula for the sum over $l$. Note that $\Zag_{l}(1/2)$ is a Fourier coefficient of a combination of Eisenstein-Maass series of half-integral weight and level $4$.  Evaluating the sum over $a$ we get either Gauss sums or half-integral weight Kloosterman sums. In both cases, changing the orders of summations over $c$ and $q$, and then evaluating  the sum over $c$, we obtain Dirichlet $L$-functions of quadratic characters. Since Zagier $L$-series are generalizations of Dirichlet $L$-functions, we finally obtain the second moment of Dirichlet $L$-functions as in the proof of Khan and Young. Estimating it with the help of Heath-Brown large sieve for quadratic characters completes the proof of Theorem \ref{KY result}.

Our method is difficult to compare with the one of Khan and Young due to presence of several changes of variables in our proof. The advantage of our approach is a simplification in the analysis of the arithmetic part of the proof.  More precisely, instead of quite complicated formula for $T(k,l,c)$ (see \cite[Lemma 5.3]{KhYoung}) we obtain some nice generalization of Dirichlet $L$-functions of  quadratic characters, namely Zagier $L$-series $\Zag_{n^2-4m^2}(1/2)$. Furthermore, in the analytic part of the proof we avoid any multiple integrals and extract all desired estimates from the asymptotic behaviour of integrals of hypergeometric functions.

\section{Notation and preliminaries}\label{sec:prelim}
Let $e(x):=\exp(2\pi ix).$ Let $\Gamma(z)$ be the Gamma function. According to Stirling's formula
\begin{multline}\label{Stirling2}
\Gamma(\sigma+it)=\sqrt{2\pi}|t|^{\sigma-1/2}\exp(-\pi|t|/2)
\exp\left(i\left(t\log|t|-t+\frac{\pi t(\sigma-1/2)}{2|t|}\right)\right)\\\times
\left(1+O(|t|^{-1})\right)
\end{multline}
for $|t|\rightarrow\infty$ and a  fixed $\sigma$. We remark that one can write a full asymptotic expansion instead of  $O(|t|^{-1})$.

Let $\{u_j\}$ be an orthonormal basis of the space of Maass cusp forms of level one such that each $u_j$ is an eigenfunctions of all Hecke operators and the hyperbolic Laplacian. We denote by $\{\lambda_{j}(n)\}$ the eigenvalues of Hecke operators acting on $u_{j}$ and by $\kappa_{j}=1/4+t_{j}^2$ (with  $t_j>0$)  the eigenvalues of the hyperbolic Laplacian acting on $u_{j}$.
It is known that
\begin{equation*}
u_{j}(x+iy)=\sqrt{y}\sum_{n\neq 0}\rho_{j}(n)K_{it_j}(2\pi|n|y)e(nx),
\end{equation*}
where $K_{\alpha}(x)$ is the $K$-Bessel function and $\rho_{j}(n)=\rho_{j}(1)\lambda_{j}(n).$
The associated symmetric square $L$-function is defined (for $\Re{s}>1$) by
\begin{equation*}
L(\sym^2 u_{j},s)=\zeta(2s)\sum_{n=1}^{\infty}\frac{\lambda_{j}(n^2)}{n^s}.
\end{equation*}
This function can be  analytically  continued to the whole complex plane. It also satisfies the following functional equation:
\begin{equation}\label{func.eq}
L_{\infty}(s,t_j)L(\sym^2 u_{j},s)=L_{\infty}(1-s,t_j)L(\sym^2 u_{j},1-s),
\end{equation}
where
\begin{equation}\label{L.infinity}
L_{\infty}(s,t_j)=\pi^{-3s/2}\Gamma\left(\frac{s}{2}\right)\Gamma\left(\frac{s+2it_j}{2}\right)\Gamma\left(\frac{s-2it_j}{2}\right).
\end{equation}
As a consequence of \eqref{func.eq} one can obtain the following approximate functional equation.
\begin{lem}
We have
\begin{equation}\label{approx.func.eq.}
L(\sym^2 u_{j},1/2)=2\sum_{m=1}^{\infty}\frac{\lambda_j(m^2)}{m^{1/2}}V(n,t_j),
\end{equation}
where for any $y>0$ and $a>0$
\begin{equation}\label{approx.fun.eq.Vdef}
V(y,t_j)=\frac{1}{2\pi i}\int_{(a)}\frac{L_{\infty}(1/2+z,t_j)}{L_{\infty}(1/2,t_j)}\zeta(1+2z)G(z)y^{-z}\frac{dz}{z}
\end{equation}
and
\begin{equation}\label{Gdef}
G(z)=\exp(z^2)P_n(z^2).
\end{equation}
Here $P_n$ is a polynomial of degree n such that $P_n(0)=1$ and $P_n((1/2+2j)^2)=0$ for $j=0,\ldots, n-1$.
\end{lem}
\begin{proof}
The proof is similar to \cite[Lemma 2.2]{Khan} (see also \cite[Lemma 1]{TangXu} and \cite[Lemma 7.2.1]{Ng}).
\end{proof}
Arguing in the same way as in \cite[Lemma 2]{TangXu} and applying  \eqref{L.infinity}, \eqref{Stirling2} we prove the  results below.
\begin{lem}
For any positive $y,\,t_j$ and $A$ we have
\begin{equation}\label{Vestimate}
V(y,t_j)\ll\left(\frac{t_j}{y}\right)^{A}.
\end{equation}
For any positive integer $N$ and $1\le y\ll t_j^{1+\epsilon}$ the following asymptotic formula holds
\begin{multline}\label{Vapproximation}
V(y,t_j)=\frac{1}{2\pi i}\int_{(a)}
\left(\frac{t_j}{\pi^{3/2}y}\right)^z
\frac{\Gamma(1/4+z/2)}{\Gamma(1/4)}\zeta(1+2z)G(z)
\left(1+\sum_{k=1}^{N-1}\frac{p_{2k}(v)}{t_j^k}\right)\frac{dz}{z}+O(t_j^{-N+\epsilon}),
\end{multline}
where $v=\Im(z)$ and $p_n(v)$ is a polynomial of degree n.
\end{lem}
For $\Re{s}>1$ the Zagier $L$-series is defined by
\begin{equation}\label{Lbyk}
\Zag_{n}(s)=\frac{\zeta(2s)}{\zeta(s)}\sum_{q=1}^{\infty}\frac{\rho_q(n)}{q^{s}}=\sum_{q=1}^{\infty}\frac{\lambda_q(n)}{q^{s}},
\end{equation}
where
\begin{equation}
\rho_q(n):=\#\{x\Mod{2q}:x^2\equiv n\Mod{4q}\},\quad
\lambda_q(n):=\sum_{q_{1}^{2}q_2q_3=q}\mu(q_2)\rho_{q_3}(n).
\end{equation}
A  meromorphic continuation and a functional equation for the series  \eqref{Lbyk} were proved by Zagier, see \cite[Proposition 3]{Z}. An important  property of this series is the following
(see \cite[Proposition 3]{Z}):  $$\Zag_n(s)=0 \text{ if } n \equiv 2,3 \Mod{4}.$$
For $n=Dl^2$ with $D$ fundamental discriminant one has
\begin{equation}\label{ldecomp}
\Zag_{n}(s)=l^{1/2-s}T_{l}^{(D)}(s)L(s,\chi_D),
\end{equation}
where $L(s,\chi_D)$ is a Dirichlet $L$-function of primitive quadratic character $\chi_D$ and
\begin{equation}\label{eq:td}
T_{l}^{(D)}(s)=\sum_{l_1l_2=l}\chi_D(l_1)\frac{\mu(l_1)}{\sqrt{l_1}}\tau_{s-1/2}(l_2).
\end{equation}
It follows from \eqref{ldecomp} and the subconvexity estimates on $L(1/2+it,\chi_D)$ due to Conrey and Iwaniec \cite{CI}
and Young \cite{Y} that
\begin{equation}\label{eq:subconvexity}
\Zag_n(1/2+it)\ll (1+|n|)^{\theta}(1+|t|)^{\theta}, \quad \theta=1/6+\epsilon.
\end{equation}
\section{Voronoi's formula for $\Zag_n(1/2)$}\label{sec:Voronoi}
Here we state the Voronoi summation formula for Zagier $L$-series established in \cite{BFVoron}. To this end, it is required to introduce some notation and consider several cases.

For a matrix  $M=\begin{bmatrix}
a&b\\
c&d
\end{bmatrix}\in\Gamma_0(4)$
the theta multiplier is defined as (see \cite[p.2380]{Strom})
\begin{equation}\label{ThetaMultiplier def}
 \TM(M)=\bar{\epsilon}_{d}\left(\frac{c}{d}\right),
\end{equation}
where the symbol $\left(\frac{c}{d}\right)$ is given on \cite[p.442]{Shimura} (in the case when $d$ is a positive odd number it coincides with the classical Jacobi symbol), and $\epsilon_{d}=1$ if $d\equiv1\pmod 4$ and $\epsilon_{d}=i$ if $d\equiv-1\pmod 4$.

Let $\phi(y)$ be a smooth compactly supported function defined for $y>0$ and
\begin{equation}\label{R def}
\Res(x)=\left(\gamma-\log(4\pi)+\frac{1}{2}\right)\phi^{+}(1/2)\frac{\pi^{1/2}x^{-1}}{\Gamma(3/4)},
\end{equation}
where $\phi^{+}(s)$ is a Mellin transform of $\phi(y)$.

Let $\Phi^{(+,+)}(x)$ and $\Phi^{(-,+)}(x)$ be
\begin{equation}\label{Phi++-+def}
\Phi^{(+,+)}(x)=\frac{-x^{1/2}}{2^{1/2}}\left(Y_{0}(2\sqrt{x})+J_{0}(2\sqrt{x})\right),\quad
\Phi^{(-,+)}(x)=\frac{2\sqrt{x}K_{0}(2\sqrt{x})}{\Gamma^2(3/4)}.
\end{equation}
For $y>0$  we introduce the integral transforms:
\begin{equation}\label{phi hat+0 to Phipm def0}
\widehat{\phi}(y)=\int_0^{\infty}\frac{\phi(x)}{x}\Phi^{(+,+)}(xy)dx,\quad
\widehat{\phi}(-y)=\int_0^{\infty}\frac{\phi(x)}{x}\Phi^{(-,+)}(xy)dx.
\end{equation}
Finally, let
\begin{equation*}
a_{n}=\frac{\Gamma\left(\frac{1}{2}-\frac{\sgn{n}}{4}\right)\Zag_{n}\left(\frac{1}{2}\right)}{2\sqrt{\pi|n|}}.
\end{equation*}

\begin{lem}\label{Thm Voronoi an c0mod4}
For $c\equiv0\pmod{4}$ and $ad\equiv1\pmod{4}$ one has
\begin{equation}\label{eq:mainresult c4}
\sum_{n=-\infty}^{\infty}\phi(n)a_{n}e\left( \frac{an}{c}\right)=
\TM(M_0)e(1/8)\Res(c)+
\TM(M_0)e(1/8)\sum_{n\neq0}\widehat{\phi}\left(\frac{4\pi^2n}{c^2}\right)a_{n}e\left( -\frac{dn}{c}\right),
\end{equation}
where $\TM(M_0)=\bar{\epsilon}_{d}\left(\frac{c}{d}\right)$.
\end{lem}
\begin{lem}\label{Thm Voronoi an codd}
For $(c,2)=1$ and $4ad\equiv -1\pmod{c}$ one has
\begin{equation}\label{Thm.eq Voronoi an codd}
\sum_{n=-\infty}^{\infty}\phi(n)a_{n}e\left( \frac{an}{c}\right)=
\TM(M_1)\sqrt{2}\Res(4c)+
\TM(M_1)\sqrt{2}\sum_{n\neq0}\widehat{\phi}\left(\frac{\pi^2n}{c^2}\right)a_{4n}e\left(\frac{dn}{c}\right),
\end{equation}
where $\TM(M_1)=\bar{\epsilon}_{c}\left(\frac{4d}{c}\right)$.
\end{lem}


\begin{lem}\label{Thm Voronoi an c2mod4}
For  $c\equiv2\pmod{4}$ let $c_1=c/2$. For $2ad_5\equiv-1\pmod{c_1}$,  $8ad_4\equiv-1\pmod{c_1}$ one has
\begin{multline}\label{Thm.eq Voronoi an c2mod4}
\sum_{n=-\infty}^{\infty}\phi(n)a_{n}e\left( \frac{an}{c}\right)=
\frac{\sqrt{2}}{\TM(M_4)}\Res(4c_1)-
\TM(M_5)\sqrt{2}\Res(4c_1)\\+
\frac{\sqrt{2}}{\TM(M_4)}\sum_{n\neq0}
\widehat{\phi}\left(\frac{\pi^2n}{4c_1^2}\right)a_{n}e\left(\frac{d_4n}{c_1}\right)-
\TM(M_5)\sqrt{2}\sum_{n\neq0}\widehat{\phi}\left(\frac{\pi^2n}{c_1^2}\right)a_{4n}e\left(\frac{d_5n}{c_1}\right),
\end{multline}
where $\TM^{-1}(M_4)=\epsilon_{c_1}\left(\frac{8a}{c_1}\right)$ and $\TM(M_5)=\bar{\epsilon}_{c_1}\left(\frac{4d_5}{c_1}\right)$.
\end{lem}
To avoid misunderstanding in evaluating inverse elements, let $\overline{a}_q$ be an inverse of $a$ modulo $q$  that is  $\overline{a}_qa\equiv 1\pmod{q}$.

After applying the Voronoi formula we will sum the result over $a\pmod{c}$, $(a,c)=1$.

For $c\equiv0\pmod{4}$ and $ad\equiv1\pmod{4}$   let
\begin{equation}\label{STM 0 def}
\STM_0(c,n):=\sum_{\substack{a\pmod{c}\\(a,c)=1}}\TM(M_0)e\left( -\frac{dn}{c}\right)=
\sum_{\substack{a\pmod{c}\\(a,c)=1}}\bar{\epsilon}_{d}\left(\frac{c}{d}\right)e\left( -\frac{dn}{c}\right).
\end{equation}
For $(c,2)=1$ and $4ad\equiv -1\pmod{c}$ let
\begin{multline}\label{STM 1 def}
\STM_1(c,n):=\sum_{\substack{a\pmod{c}\\(a,c)=1}}\TM(M_1)e\left(\frac{dn}{c}\right)=
\bar{\epsilon}_{c}\sum_{\substack{a\pmod{c}\\(a,c)=1}}\left(\frac{4d}{c}\right)e\left(\frac{dn}{c}\right)=\\=
\bar{\epsilon}_{c}\sum_{\substack{a\pmod{c}\\(a,c)=1}}\left(\frac{-\overline{4a}_c}{c}\right)e\left(\frac{-\overline{4a}_cn}{c}\right)=
\bar{\epsilon}_{c}\sum_{\substack{a\pmod{c}\\(a,c)=1}}\left(\frac{a}{c}\right)e\left(\frac{an}{c}\right).
\end{multline}
For  $c\equiv2\pmod{4}$, $c_1=c/2$,  $8ad_4\equiv-1\pmod{c_1}$ and $2ad_5\equiv-1\pmod{c_1}$ let
\begin{equation}\label{STM 4 def}
\STM_4(c,n):=\sum_{\substack{a\pmod{c}\\(a,c)=1}}\TM^{-1}(M_4)e\left(\frac{d_4n}{c_1}\right)=
\epsilon_{c_1}\sum_{\substack{a\pmod{c}\\(a,c)=1}}\left(\frac{8a}{c_1}\right)
e\left(\frac{-\overline{8a}_{c_1}n}{c_1}\right),
\end{equation}
\begin{multline}\label{STM 5 def}
\STM_5(c,n):=\sum_{\substack{a\pmod{c}\\(a,c)=1}}\TM(M_5)e\left(\frac{d_5n}{c_1}\right)=
\bar{\epsilon}_{c_1}\sum_{\substack{a\pmod{c}\\(a,c)=1}}
\left(\frac{4d_5}{c_1}\right)e\left(\frac{d_5n}{c_1}\right)=\\=
\bar{\epsilon}_{c_1}\sum_{\substack{a\pmod{c}\\(a,c)=1}}
\left(\frac{-2\overline{a}_{c_1}}{c_1}\right)e\left(\frac{-\overline{2a}_{c_1}n}{c_1}\right).
\end{multline}
\begin{lem}\label{Lemma STM45}
For  $c\equiv2\pmod{4}$ one has
\begin{equation}\label{STM 45 to STM1}
\STM_4(c,n)=\STM_5(c,n)=\STM_1(c/2,n).
\end{equation}
\end{lem}
\begin{proof}
The set $S=\{a\pmod{c},\, (a,c)=1\}$ being considered modulo $c_1$ coincides with the set $\{a\pmod{c_1},\, (a,c_1)=1\}$.  This is because both of the sets consist of the same number of elements $\phi(c)=\phi(2c_1)=\phi(c_1)$, and all  $a\pmod{c}$, $(a,c)=1$ have different remainders modulo $c_1$. Indeed, suppose that $a_1\equiv a_2\pmod{c_1}$. Then since $a_1$, $a_2$ are odd one also has $a_1\equiv a_2\pmod{2}$ and thus  $a_1\equiv a_2\pmod{c}$ a contradiction.

Therefore, \eqref{STM 4 def} can be rewritten as
\begin{multline}\label{STM 4 def2}
\STM_4(c,n)=
\epsilon_{c_1}\sum_{\substack{a\pmod{c_1}\\(a,c_1)=1}}\left(\frac{8a}{c_1}\right)e\left(\frac{-\overline{8a}_{c_1}n}{c_1}\right)=
\epsilon_{c_1}\sum_{\substack{a\pmod{c_1}\\(a,c_1)=1}}\left(\frac{a}{c_1}\right)e\left(\frac{-\overline{a}_{c_1}n}{c_1}\right)=\\=
\epsilon_{c_1}\left(\frac{-1}{c_1}\right)\sum_{\substack{a\pmod{c_1}\\(a,c_1)=1}}\left(\frac{a}{c_1}\right)e\left(\frac{an}{c_1}\right)=
\bar{\epsilon}_{c_1}\sum_{\substack{a\pmod{c_1}\\(a,c_1)=1}}\left(\frac{a}{c_1}\right)e\left(\frac{an}{c_1}\right).
\end{multline}
In the same way \eqref{STM 5 def} transforms into the form:
\begin{equation}\label{STM 5 def2}
\STM_5(c,n)=
\bar{\epsilon}_{c_1}\sum_{\substack{a\pmod{c}\\(a,c)=1}}\left(\frac{-2\overline{a}_{c_1}}{c_1}\right)e\left(\frac{-\overline{2a}_{c_1}n}{c_1}\right)=
\bar{\epsilon}_{c_1}\sum_{\substack{a\pmod{c_1}\\(a,c_1)=1}}\left(\frac{a}{c_1}\right)e\left(\frac{an}{c_1}\right).
\end{equation}
\end{proof}

Next, let us investigate the series:
\begin{equation}\label{STM series}
\sum_{c\equiv0\pmod{4}}\frac{\STM_0(c,n)}{c^{2s}},\quad
\sum_{(c,2)=1}\frac{\STM_1(c,n)}{c^{2s}}.
\end{equation}
To this end, we first define several objects studied in \cite{PRR}.  Let $n=n_0n_1^2$ with $n_0$ being square-free. Define (see \cite[(2.7)]{PRR})
\begin{equation}\label{PRR q def}
q(w,n):=\prod_{\substack{p|n_1\\p\neq2}}\sum_{\beta=0}^{\nu_p(n_1)}\frac{1-\delta_{\beta<\nu_p(n_1)}\chi_{n_0}(p)p^{-w}}{p^{2\beta(w-1/2)}}.
\end{equation}
Here $\chi_{n_0}(p)=\left(\frac{n_0}{p}\right)$ is a Legendre symbol, $\nu_p(n_1)$ is equal to the number $k$ such that $n_1=p^kn_2$, $(n_2,p)=1$, and for some statement $S$  the symbol $\delta_S=1$ if $S$ is true and $\delta_S=0$ otherwise. The function $q(w,n)$ satisfies the following properties (see \cite[Proposition 2.3]{PRR}).
If $n$ is square-free, then $q(w,n)=1$. For $\Re{s}\ge 1/2$ one has $q(w,n)\ll n^{\epsilon}$. Also $q(w,n)=n_1^{1-2w}q(1-w,n).$

Let (see \cite[(2.8)]{PRR})
\begin{equation}\label{PRR Last def}
L^{\ast}(w,n):=q(w,n)L_2(w,\chi_{n_0}),
\end{equation}
where $L_2(w,\chi_{n_0})$ is the standard Dirichlet $L$-function with the 2-factor removed.
Using \cite[(2.3), p.1546 line 5]{PRR}, we obtain $\STM_1(c,n)=H_n(c).$ Thus \cite[Lemma 2.2]{PRR} yields
\begin{equation}\label{STM1 series}
\sum_{(c,2)=1}\frac{\STM_1(c,n)}{c^{2s}}=\frac{L^{\ast}(2s-1/2,n)}{\zeta_2(4s-1)}.
\end{equation}
Note that $q(w,n)$ is closely related to $l^{1/2-w}T_{n_1}^{(n_0)}(w)$ (see \eqref{eq:td}), since
\begin{equation}\label{PRR q def2}
q(w,n)=\prod_{\substack{p|n_1\\p\neq2}}\left(
\frac{1-p^{(1-2w)(1+\nu_p(n_1))}}{1-p^{1-2w}}-\frac{\chi_{n_0}(p)}{p^w}\frac{1-p^{(1-2w)\nu_p(n_1)}}{1-p^{1-2w}}
\right)
\end{equation}
and
\begin{equation}\label{eq:td2}
l^{1/2-w}T_{n_1}^{(n_0)}(w)=\prod_{p|n_1}\left(
\frac{1-p^{(1-2w)(1+\nu_p(n_1))}}{1-p^{1-2w}}-\frac{\chi_{n_0}(p)}{p^w}\frac{1-p^{(1-2w)\nu_p(n_1)}}{1-p^{1-2w}}
\right).
\end{equation}
Therefore, $L^{\ast}(w,n)$ and $\Zag_n(w)$ differ slightly from each other.
It follows from \cite[(3.1),(3.3)]{PRR} that
\begin{equation}\label{STM0 series}
\sum_{c\equiv0\pmod{4}}\frac{\overline{\STM_0(c,n)}}{c^{2s}}=\frac{L^{\ast}(2s-1/2,n)}{\zeta_2(4s-1)}r_2(s,n),
\end{equation}
where (see \cite[(3.6),(3.7)]{PRR})
\begin{equation}\label{r2 def23}
r_2(s,n)=\frac{-(1+i)}{2^{2s}},\quad\,n\equiv2,3\pmod{4},
\end{equation}
\begin{equation}\label{r2 def1}
r_2(s,n)=\frac{(1+i)}{2^{2s}}+\frac{(1+i)(-1)^{(n-1)/4}}{2^{6s-1/2}},\quad\,n\equiv1\pmod{4},
\end{equation}
and for $n=4^{r}n_4$ with $n_4\neq0\pmod{4}$
\begin{equation}\label{r2 def0}
r_2(s,n)=\frac{(1+i)}{4}u_r(2^{1-2s})+4^{r(1-2s)}r_2(s,n_4),\quad  u_r(x)=\frac{x^{2(r+1)}-x^2}{x^2-1}.
\end{equation}
One can rewrite \eqref{STM0 series} in the form
\begin{equation}\label{STM0 series2}
\sum_{c\equiv0\pmod{4}}\frac{\STM_0(c,n)}{c^{2s}}=\frac{L^{\ast}(2s-1/2,n)}{\zeta_2(4s-1)}\overline{r_2(\bar{s},n)}.
\end{equation}
Using Heath-Brown's large sieve inequality \cite{HB} (see \cite[(3.1)]{KhYoung}), one has \cite[(3.2)]{KhYoung}
\begin{equation}\label{Last 2mom estimate}
\sum_{n\le N}|L^{\ast}(1/2+it,n)|^2\ll\left(N+\sqrt{N(1+|t|)}\right)\left(N(1+|t|)\right)^{\epsilon},
\end{equation}
\begin{equation}\label{LZag 2mom estimate}
\sum_{n\le N}|\Zag_n(1/2+it)|^2\ll\left(N+\sqrt{N(1+|t|)}\right)\left(N(1+|t|)\right)^{\epsilon}.
\end{equation}

In order to find an asymptotic expansion for $\widehat{\phi}(y)$ we will use the following version of the saddle point method due to Huxley \cite{Hux}.

\begin{lem}\label{Lemma Huxley}
Let
\begin{equation}\label{I Hux def}
I=\int_a^{b}g(x)e(f(x))dx,
\end{equation}
where
$f(x)$ and $g(x)$ are some smooth functions defined on the interval $[a,b]$ such that $g(a)=g(b)=0$ and
\begin{equation}\label{Hux fg conditions}
f^{(i)}(x)\ll\frac{\Theta_f}{\Omega_f^i},\quad
g^{(j)}(x)\ll\frac{1}{\Omega_g^j},\quad
f^{(2)}(x)\gg\frac{\Theta_f}{\Omega_f^2}
\end{equation}
for $i=1,2,3,4$ and $j=0,1,2$. Suppose that $f'(x_0)=0$ for $a<x_0<b$, and that $f(x)$ changes its sign from negative to positive at $x=x_0$. Let $\kappa=\min(x_0-a,b-x_0)$. Then
\begin{equation}\label{I Hux asympt}
I=\frac{g(x_0)e(f(x_0)+1/8)}{\sqrt{f''(x_0)}}+O\left(\frac{\Omega_f^4}{\kappa^3\Theta_f^2}+
\frac{\Omega_f}{\Theta_f^{3/2}}+\frac{\Omega_f^3}{\Theta_f^{3/2}\Omega_g^2}\right).
\end{equation}
\end{lem}

To bound various integrals we will usually apply estimates based on  multiple integration by parts, namely  \cite[Lemma 8.1]{BKY} (see also \cite[Lemma A1]{AHLQ}).
\begin{lem}\label{Lemma BKY}
Suppose that there are parameters $R,P,X,Y,V$ such that
\begin{equation}\label{BKYconditions}
|f'(x)|\gg R,\quad
f^{(i)}(x)\ll\frac{Y}{P^i},\quad
g^{(j)}(x)\ll\frac{X}{V^j}
\end{equation}
for $i\ge1,j\ge0$. Then
\begin{equation}\label{I BKY est}
I=\int_a^{b}g(x)e(f(x))dx\ll(b-a)X\left(\frac{1}{RV}+\frac{1}{RP}+\frac{Y}{R^2P^2}\right)^{A}.
\end{equation}
\end{lem}




\section{The second moment and the sum of Zagier $L$-series}
For an arbitrary large integer $N$, we define
\begin{equation}\label{qN def}
q_N(r):=\frac{(r^2+1/4)\ldots(r^2+(N-1/2)^2)}{(r^2+100N^2)^N},
\end{equation}
\begin{equation}\label{hN def}
h(T,G,N;r):=q_N(r)\exp\left(-\frac{(r-T)^2}{G^2}\right)+q_N(r)\exp\left(-\frac{(r+T)^2}{G^2}\right).
\end{equation}
Note that $q_N(r)=1+O(1/r^2)$. For the sake of simplicity, we will write $h(r)$ instead of $h(T,G,N;r)$.

In order to prove Theorem \ref{KY result} we consider the second moment:
\begin{equation}
\M_2:=\sum_{j}h(t_j)\alpha_{j}L^2(\sym^2 u_{j},1/2).
\end{equation}
Using the results of \cite[sec.5]{BF2mom} with $l=1$ we infer
\begin{equation}\label{M2 est1}
\M_2\ll T^{1+\epsilon}G+\sum_{m\ll T^{1+\epsilon}}\frac{1}{m^{1/2}}
\sum_{n=2m+1}^{\infty}\frac{1}{n^{1/2}}\Zag_{n^2-4m^2}(1/2)I\left(m,\frac{n}{m}\right),
\end{equation}
where
\begin{multline}\label{I(m,x)def}
I(m,x)=\int_{-\infty}^{\infty}\frac{rh(r)V(m,r)}{\cosh(\pi r)}\left(\frac{2}{x}\right)^{2ir}
\frac{\Gamma(1/4+ir)\Gamma(3/4+ir)}{\Gamma(1+2ir)}\\ \times \sin\left( \pi(1/4-ir)\right) {}_2F_{1}\left(1/4+ir,3/4+ir,1+2ir;\frac{4}{x^2} \right)dr
\end{multline}
and $V(m,r)$ is given by \eqref{approx.fun.eq.Vdef}. As in the beginning of the proof of  \cite[Lemma 5.2]{BF2mom}, we first  truncate the series over $n$ introducing a negligible error term, and then apply an asymptotic expansion  \cite[Lemma 4.6]{BF2mom} for the hypergeometric function in \eqref{I(m,x)def}. Computations in
\cite[Lemma 5.2]{BF2mom} show that we can truncate the sum at $n\ll2m+T^{\epsilon}m/G^2$ with negligibly small error term.
If $n\gg2m+T^{\epsilon}m/G^2$ then (see \cite[(5.19)]{BF2mom}) one has
\begin{equation}
G^2\arcosh^2\frac{n}{2m}\gg G^2\frac{|n-2m|}{2m}\gg T^{\epsilon}.
\end{equation}
Therefore,
\begin{equation}\label{M2 est2}
\M_2\ll T^{1+\epsilon}G+\sum_{m\ll T^{1+\epsilon}}\frac{1}{m^{1/2}}
\sum_{0<n-2m\ll T^{\epsilon}m/G^2}\frac{1}{n^{1/2}}\Zag_{n^2-4m^2}(1/2)I\left(m,\frac{n}{m}\right).
\end{equation}
Furthermore, the following estimate (see \cite[(5.6)]{BF2mom}) holds:
\begin{equation}\label{M2 est3.0}
\sum_{m\ll T^{1+\epsilon}}\frac{1}{m^{1/2}}\sum_{2m<n\ll 2m+T^{\epsilon}m/G^2}\frac{1}{n^{1/2}}\Zag_{n^2-4m^2}(1/2)I\left(m,\frac{n}{m}\right)\ll
\frac{T^{3/2+2\theta+\epsilon}}{G^{1/2+2\theta}}.
\end{equation}
We would like to simplify $I\left(m,\frac{n}{m}\right)$ on the right-hand side of \eqref{M2 est2}.
In view of \eqref{M2 est3.0},  it is sufficient to consider only the main term in the asymptotic expansion of $I\left(m,\frac{n}{m}\right)$ in terms of $T$ because the next term is smaller than $T^{1/2+2\theta+\epsilon}/G^{1/2+2\theta}\ll T^{1+\epsilon}G$. For simplicity, we keep the notation $I(m,x)$ for the main term without changing it at each step of the proof. Applying \cite[Lemma 4.6]{BF2mom} we show that
\begin{multline}\label{I(m,x)eq1}
I(m,x)=\left(\frac{x^2}{x^2-4}\right)^{1/4}\int_{-\infty}^{\infty}\frac{rh(r)V(m,r)}{\cosh(\pi r)}2^{2ir}
\frac{\Gamma(1/4+ir)\Gamma(3/4+ir)}{\Gamma(1+2ir)}\\ \times \sin\left( \pi(1/4-ir)\right)
\exp(-2ir\arcosh(x/2))dr.
\end{multline}
Using \eqref{Stirling2} we have
\begin{equation}\label{I(m,x)eq2}
I(m,x)=\left(\frac{x^2}{x^2-4}\right)^{1/4}\int_{-\infty}^{\infty}|r|^{1/2}h(r)V(m,r)
\exp(-2ir\arcosh(x/2))dr.
\end{equation}
Let $A(x):=\arcosh{x/2}.$ Making the change of variables $r=T+Gy$ and using \eqref{hN def}, we show that
\begin{equation}\label{I(m,x)eq3}
I(m,x)=\left(\frac{x^2}{x^2-4}\right)^{1/4}
GT^{1/2}\exp(-2iTA(x))\int_{-\infty}^{\infty}\exp(-y^2-2iGyA(x))V(m,T+Gy)dy.
\end{equation}
The integral above was evaluated in \cite[(5.13)]{BF2mom}, where the right-hand side is equal to $V(m,T)$, and therefore,
\begin{equation}\label{I(m,x)eq4}
I(m,x)=\left(\frac{x^2}{x^2-4}\right)^{1/4}
GT^{1/2}\exp(-2iTA(x))\exp\left(-G^2A(x)^2\right)V(m,T).
\end{equation}
Since for $n>4m$ we have $G^2A(n/m)\gg G^2$ and for  $2m<n<4m$
\begin{equation}
G^2A(n/m)=G^2\arcosh^2\frac{n}{2m}>\frac{n-2m}{2m}G^2,
\end{equation}
we conclude that $I(m,n/m)$ is negligible unless $0<n-2m\ll T^{\epsilon}m/G^2$.
Furthermore, it is convenient to remove the factor $GT^{1/2}\left(\frac{x^2}{x^2-4}\right)^{1/4}$  in \eqref{I(m,x)eq4}, thus obtaining
\begin{equation}\label{M2 est3}
\M_2\ll T^{1+\epsilon}G+GT^{1/2}\sum_{m\ll T^{1+\epsilon}}\frac{1}{m^{1/2}}
\sum_{0<n-2m\ll T^{\epsilon}m/G^2}\frac{\Zag_{n^2-4m^2}(1/2)}{(n^2-4m^2)^{1/4}}I\left(m,\frac{n}{m}\right),
\end{equation}
where
\begin{equation}\label{I(m,x)eq5}
I(m,x)=\exp(-2iT\arcosh(x/2))\exp\left(-G^2\arcosh(x/2)^2\right)V(m,T).
\end{equation}
Note that if $G\gg T^{1/2+\epsilon}$ the sum over $n$ in \eqref{M2 est3} is empty and we prove \eqref{mean Lindelof}. Hence from now on we assume that $G\ll T^{1/2+\epsilon}$. Let
\begin{equation}\label{M0 def}
M_0:=T^{1/(2+4\theta)-\epsilon}G^{(3+4\theta)/(2+4\theta)}.
\end{equation}
Applying \eqref{eq:subconvexity} we prove that
\begin{equation}\label{M2 est4}
GT^{1/2}\sum_{m\ll M_0}\frac{1}{m^{1/2}}
\sum_{0<n-2m\ll T^{\epsilon}m/G^2}\frac{\Zag_{n^2-4m^2}(1/2)}{(n^2-4m^2)^{1/4}}I\left(m,\frac{n}{m}\right)\ll
T^{1+\epsilon}G.
\end{equation}
Consequently,
\begin{equation}\label{M2 est5}
\M_2\ll T^{1+\epsilon}G+GT^{1/2}\sum_{M_0<m\ll T^{1+\epsilon}}\frac{1}{m^{1/2}}
\sum_{0<n-2m\ll T^{\epsilon}m/G^2}\frac{\Zag_{n^2-4m^2}(1/2)}{(n^2-4m^2)^{1/4}}I\left(m,\frac{n}{m}\right).
\end{equation}
Note that $M_0\ll T^{1+\epsilon}$ only if $G\ll T^{(1+4\theta)/(3+4\theta)+\epsilon}$, which we assume from now on.
Let
\begin{equation}\label{alpha def}
\al:=\frac{1-4\theta}{3+4\theta}.
\end{equation}
Using \eqref{eq:subconvexity} we have
\begin{equation}\label{M2 est6}
GT^{1/2}\sum_{M_0<m\ll T^{1+\epsilon}}\frac{1}{m^{1/2}}
\sum_{0<n-2m\ll T^{\al}}\frac{\Zag_{n^2-4m^2}(1/2)}{(n^2-4m^2)^{1/4}}I\left(m,\frac{n}{m}\right)\ll
T^{1+\epsilon}G,
\end{equation}
which yields
\begin{equation}\label{M2 est6}
\M_2\ll T^{1+\epsilon}G+GT^{1/2}\sum_{M_0<m\ll T^{1+\epsilon}}\frac{1}{m^{1/2}}
\sum_{T^{\al}<n-2m\ll T^{\epsilon}m/G^2}\frac{\Zag_{n^2-4m^2}(1/2)}{(n^2-4m^2)^{1/4}}I\left(m,\frac{n}{m}\right).
\end{equation}
Now we make the following change of variables in \eqref{M2 est6}:
\begin{equation}\label{change of variables1}
n-2m=q,\quad
n+2m=r,
\end{equation}
showing that
\begin{equation}\label{M2 est7}
\M_2\ll T^{1+\epsilon}G+GT^{1/2}\sum_{T^{\al}<q\ll T^{1+\epsilon}/G^2}
\sum_{\substack{qG^{2-\epsilon}\ll r\ll T^{1+\epsilon}\\ r\equiv q\pmod{4}}}
\frac{\Zag_{qr}(1/2)}{(qr)^{1/4}(r-q)^{1/2}}I\left(\frac{r-q}{4},2\frac{r+q}{r-q}\right).
\end{equation}
Note that we can also add the condition  $r>q+4M_0$. First, we would like to remove the condition $r\equiv q\pmod{4}$ in \eqref{M2 est7}.  To do this, we rely on the fact that $\Zag_{qr}(s)=0$ if $qr \equiv 2,3 \Mod{4}.$ Therefore, for any coefficient $f(q,r)$ the following equality holds:
\begin{multline}\label{removing r=q(4)}
\sum_{\substack{q,r=1\\r\equiv q\pmod{4}}}^{\infty}\Zag_{qr}(1/2)f(q,r)=
\sum_{q,r=1}^{\infty}\Zag_{qr}(1/2)f(q,r)-
\sum_{q=1}^{\infty}\sum_{\substack{r=1\\r\equiv 0\pmod{4}}}^{\infty}\Zag_{qr}(1/2)f(q,r)\\-
\sum_{r=1}^{\infty}\sum_{\substack{q=1\\q\equiv 0\pmod{4}}}^{\infty}\Zag_{qr}(1/2)f(q,r)+
2\sum_{\substack{q=1\\q\equiv 0\pmod{4}}}^{\infty}\sum_{\substack{r=1\\r\equiv 0\pmod{4}}}^{\infty}\Zag_{qr}(1/2)f(q,r).
\end{multline}
Applying \eqref{removing r=q(4)} we decompose \eqref{M2 est7}, proving that
\begin{equation}\label{M2 to DS}
\M_2\ll T^{1+\epsilon}G+GT^{1/2}\left(\DS_{\ast,\ast}-\DS_{\ast,4}-\DS_{4,\ast}+2\DS_{4,4}\right),
\end{equation}
where
\begin{equation}\label{DS def}
\DS_{a,b}=\sum_{\substack{T^{\al}<q\ll T^{1+\epsilon}/G^2\\ q\equiv a\pmod{4}\\}}
\sum_{\substack{qG^{2-\epsilon}\ll r\ll T^{1+\epsilon}\\ r\equiv b\pmod{4}}}
\frac{\Zag_{qr}(1/2)}{(qr)^{1/4}(r-q)^{1/2}}I\left(\frac{r-q}{4},2\frac{r+q}{r-q}\right)
\end{equation}
and the condition $n\equiv\ast\pmod{4}$ means that we are summing over all $n$.
We will consider only the first double sum  $\DS_{\ast,\ast}$. The three remaining sums can be treated in the same way.

Note that the sum over $q$ in \eqref{DS def} is much shorter than the one over $r$. So we perform one more change of variables $r=l/q$, showing that
\begin{equation}\label{DSaa eq1}
\DS_{\ast,\ast}=\sum_{T^{\al}<q\ll T^{1+\epsilon}/G^2}
\sum_{\substack{q^2G^{2-\epsilon}\ll l\ll qT^{1+\epsilon}\\ l\equiv 0\pmod{q}}}
\frac{\Zag_{l}(1/2)}{l^{1/4}(l/q-q)^{1/2}}I\left(\frac{l/q-q}{4},2\frac{l/q+q}{l/q-q}\right).
\end{equation}
Now we express the condition $l\equiv 0\pmod{q}$ in terms of additive harmonics:
\begin{equation}\label{deltaq(l)}
\delta_q(l)=\frac{1}{q}\sum_{c|q}\sum_{\substack{a\pmod{c}\\(a,c)=1}}e\left(\frac{al}{c}\right),
\end{equation}
which yields the equality
\begin{equation}\label{DSaa eq2.1}
\DS_{\ast,\ast}=\sum_{T^{\al}<q\ll T^{1+\epsilon}/G^2}\frac{1}{q^{1/2}}\sum_{c|q}\sum_{\substack{a\pmod{c}\\(a,c)=1}}
\sum_{q^2G^{2-\epsilon}\ll l\ll qT^{1+\epsilon}}
\frac{\Zag_{l}(1/2)e\left(\frac{al}{c}\right)}{l^{1/4}(l-q^2)^{1/2}}I\left(\frac{l-q^2}{4q},2\frac{l+q^2}{l-q^2}\right).
\end{equation}
Changing the order of summation we have
\begin{multline}\label{DSaa eq2}
\DS_{\ast,\ast}=\sum_{c\ll T^{1+\epsilon}/G^2}
\sum_{T^{\al}/c<q\ll T^{1+\epsilon}/cG^2}
\frac{1}{(qc)^{1/2}}\sum_{\substack{a\pmod{c}\\(a,c)=1}}
\sum_{(qc)^2G^{2-\epsilon}\ll l\ll (qC)T^{1+\epsilon}}\Zag_{l}(1/2)\\\times
\frac{e\left(\frac{al}{c}\right)}{l^{1/4}(l-(qc)^2)^{1/2}}I\left(\frac{l-(qc)^2}{4qc},2\frac{l+(qc)^2}{l-(qc)^2}\right).
\end{multline}
Next, we apply the Voronoi summation formula to the sum over $l$ in \eqref{DSaa eq2}. To simplify some computations, we first decompose the sums over $c$, $q$ and $l$ into dyadic intervals:
\begin{equation}\label{DSaa dyadic}
\DS_{\ast,\ast}=\sum_{C\ll T^{1+\epsilon}/G^2}\sum_{T^{\al}/C\ll Q\ll T^{1+\epsilon}/CG^2}\sum_{(QC)^2G^{2-\epsilon}\ll L\ll QCT^{1+\epsilon}}
\DS_{\ast,\ast}(C,Q,L),
\end{equation}
\begin{equation}\label{DSaa CQL}
\DS_{\ast,\ast}(C,Q,L)=\sum_{c}U(c/C)\sum_{q}\frac{U(q/Q)}{(cq)^{1/2}}\sum_{\substack{a\pmod{c}\\(a,c)=1}}
\sum_{l}\frac{\Zag_{l}(1/2)}{l^{1/2}}\phi(c,q,l)e\left(\frac{al}{c}\right),
\end{equation}
where
\begin{equation}\label{DSaa phi def}
\phi(c,q,l)=\frac{U(l/L)l^{1/4}}{(l-(cq)^2)^{1/2}}I\left(\frac{l-(cq)^2}{4cq},2\frac{l+(cq)^2}{l-(cq)^2}\right).
\end{equation}
From now on we  assume  that $C,Q,L$ satisfy the conditions in  \eqref{DSaa dyadic}.

\begin{rema}\label{rem phi cases}
Arguing in the same way, we prove an analogue of \eqref{DSaa CQL} for $\DS_{\ast,4}(C,Q,L)$ with $\phi(c,q/(c,4),l)$, while for $\DS_{4,\ast}(C,Q,L)$ we  get $\phi(c,4q/(c,4),l)$ and for $\DS_{4,4}(C,Q,L)$  we get $\phi(c,4q/(c,16),l)$.
\end{rema}

\section{Integral transform in Voronoi's summation formula}
First, we evaluate and estimate the integral transforms of $\phi(y)$ that appear on the right-hand side of the Voronoi formula \eqref{phi hat+0 to Phipm def0}.
\begin{lem}
The following estimate holds:
\begin{equation}\label{phi+ est0}
\phi^{+}(1/2):=\int_0^{\infty}\frac{\phi(y)}{\sqrt{y}}dy\ll T^{-A}
\end{equation}
\end{lem}
\begin{proof}
Let
\begin{equation}\label{Ay def}
A(y):=\arcosh\left(\frac{yL+(cq)^2}{yL-(cq)^2}\right).
\end{equation}
Then using \eqref{DSaa phi def} and \eqref{I(m,x)eq5} we obtain
\begin{equation}\label{phi+ est1}
\phi^{+}(1/2)=L^{3/4}\int_0^{\infty}\frac{U(y)\exp(-2iTA(y))\exp\left(-G^2A^2(y)\right)}{y^{1/4}(yL-(cq)^2)^{1/2}}
V\left(\frac{yL-(cq)^2}{4cq},T\right)dy.
\end{equation}
To estimate this integral, we apply Lemma \ref{Lemma BKY}.  Straightforward computations show that
\begin{equation}\label{Ay deriv}
TA'(y)=\frac{-TcqL^{1/2}}{(Ly-(cq)^2)y^{1/2}}\sim\frac{TCQ}{L^{1/2}},
\end{equation}
and thus we set $$R=\frac{TCQ}{L^{1/2}}.$$ We also let  $P=1$ and $Y=R.$ Finally, we are left to choose the parameter $V$ in Lemma \ref{Lemma BKY}.
Since $y^jV^{(j)}(y,T)\ll T^{\epsilon}$ and
\begin{equation}\label{e(G2A2) deriv}
\frac{d}{dy}\exp\left(-G^2A^2(y)\right)\ll \exp\left(-G^2A^2(y)\right)G^2A(y)A'(y)\ll T^{\epsilon}GA'(y)\ll\frac{T^{\epsilon}GCQ}{L^{1/2}}\ll T^{2\epsilon},
\end{equation}
we set $V=T^{-\epsilon}$. Finally, applying Lemma \ref{Lemma BKY} we prove \eqref{phi+ est0}.
\end{proof}
\begin{lem}\label{lem phihat large n}
Let $\alpha$ be some fixed positive number. The following estimate holds:
\begin{equation}\label{phihat- est0}
\widehat{\phi}\left(-\frac{\alpha n}{c^2}\right)\ll (nT)^{-A}.
\end{equation}
Furthermore, for $n\gg T^{\epsilon}(TQC^2/L)^2$ we have
\begin{equation}\label{phihat+ est0}
\widehat{\phi}\left(\frac{\alpha n}{c^2}\right)\ll T^{-A}.
\end{equation}
\end{lem}
\begin{proof}
It follows from \eqref{Phi++-+def} and \eqref{phi hat+0 to Phipm def0} that
\begin{multline}\label{phihat- est1}
\widehat{\phi}\left(-\frac{\alpha n}{c^2}\right)\ll\int_0^{\infty}\frac{\phi(x)\sqrt{n}}{c\sqrt{x}}
K_{0}\left(\frac{2\sqrt{\alpha xn}}{c}\right)dx\ll
\frac{\sqrt{n}}{c}\\\times
\int_0^{\infty}
\frac{U(x/L)}{x^{1/4}(x-(cq)^2)^{1/2}}I\left(\frac{x-(cq)^2}{4cq},2\frac{x+(cq)^2}{x-(cq)^2}\right)
K_{0}\left(\frac{2\sqrt{\alpha xn}}{c}\right)dx.
\end{multline}
Conditions in  \eqref{DSaa dyadic} imply that
\begin{equation}\label{K0 argument}
\frac{\sqrt{\alpha xn}}{c}\sim\frac{\sqrt{Ln}}{C}\gg G^{1-\epsilon}\sqrt{n}.
\end{equation}
Using \eqref{phihat- est1}, \cite[10.25.3]{HMF} and \eqref{K0 argument} we prove  \eqref{phihat- est0}.
Let $y=\alpha n/c^2$. It follows from \eqref{Phi++-+def} and \eqref{phi hat+0 to Phipm def0} that
\begin{equation}\label{phihat+ est1}
\widehat{\phi}\left(y\right)=\frac{-y^{1/2}}{2^{1/2}}\int_0^{\infty}\frac{\phi(x)}{\sqrt{x}}
\left(Y_{0}(2\sqrt{xy})+J_{0}(2\sqrt{xy})\right)dx.
\end{equation}
Using \eqref{DSaa phi def}, \eqref{I(m,x)eq5} and \eqref{Ay def} we obtain
\begin{equation}\label{phihat+ est2}
\widehat{\phi}\left(y\right)=\frac{-y^{1/2}L^{3/4}}{2^{1/2}}\int_0^{\infty}f(x)
\left(Y_{0}(2\sqrt{Lxy})+J_{0}(2\sqrt{Lxy})\right)dx,
\end{equation}
where
\begin{equation}\label{phihat+ f def}
f(x)=\frac{U(x)\exp(-2iTA(x))\exp\left(-G^2A^2(x)\right)}{x^{1/4}(Lx-(cq)^2)^{1/2}}V\left(\frac{xL-(cq)^2}{4cq},T\right).
\end{equation}
Integration by parts (see \cite[Lemma 6.1]{Har}) shows that
\begin{equation}\label{phihat+ est3}
\widehat{\phi}\left(y\right)\ll \frac{y^{1/2}L^{3/4}}{(Ly)^{j/2}}\int_0^{\infty}\frac{d^{j}}{dx^j}\left(f(x)\right)
x^{j/2}\left(Y_{j}(2\sqrt{Lxy})+J_{j}(2\sqrt{Lxy})\right)dx.
\end{equation}
Using \eqref{Ay deriv} and the fact that $A(x)\ll G^{\epsilon-1}$ we obtain
\begin{equation}\label{phihat+ f jderiv}
\frac{d^{j}}{dx^j}\left(f(x)\right)\ll\left(\frac{Tcq}{L^{1/2}}\right)^j.
\end{equation}
Therefore,
\begin{equation}\label{phihat+ est4}
\widehat{\phi}\left(y\right)\ll y^{1/4}L^{1/2}\left(\frac{Tcq}{Ly^{1/2}}\right)^j.
\end{equation}
\end{proof}
\begin{lem}\label{lem phihat}
For $m\ll\frac{(Tc^2q)^2}{L^2}$ and  $m\gg\frac{(Tc^2q)^2}{L^2}$ we have
\begin{equation}\label{phihat asympt00}
\widehat{\phi}\left(\frac{m}{c^2}\right)\ll T^{-A}.
\end{equation}
For $\sqrt{m}\sim\frac{Tc^2q}{L}$ the asymptotic formula holds
\begin{equation}\label{phihat asympt0}
\widehat{\phi}\left(\frac{m}{c^2}\right)=
\frac{\exp\left(2ih(q,m)\right)}{L^{1/4}}F(c,q;m)
+O\left(\frac{L^{1/4}T^{\epsilon}}{Tcq}\right),
\end{equation}
where
\begin{equation}\label{phihat F def}
F(c,q;m)=\frac{U\left(x_0\right)}{x_0^{1/4}}
\exp\left(-G^2\arcosh^2\left(1+\frac{q\sqrt{m}}{T}\right)\right)
V\left(\frac{cT}{2\sqrt{m}},T\right)W(Lx_0m/c^2)
\end{equation}
with   $W(z)\ll 1,$  $W^{(j)}(z)\ll z^{-1-j}$  and
\begin{equation}\label{x0 def}
x_0=\frac{2Tqc^2}{L\sqrt{m}}+\frac{(cq)^2}{L},
\end{equation}
\begin{equation}\label{h(q,n) def}
h(q,m)=-T\arcosh\left(1+\frac{q\sqrt{m}}{T}\right)-
\sqrt{2Tq\sqrt{m}+mq^2}.
\end{equation}
\end{lem}
\begin{proof}
Let $y=m/c^2.$ Since $\sqrt{Lxy}\gg L^{1/2}/C\gg G^{1-\epsilon}$, the argument of Bessel functions in \eqref{phihat+ est2} is large and we are able to use \cite[8.451]{GR}. Accordingly,
\begin{multline}\label{Y0+J0}
Y_{0}(2\sqrt{Lxy})+J_{0}(2\sqrt{Lxy})=\\=\frac{2^{1/2}}{\pi^{1/2}(Lxy)^{1/4}}\left(
\sin(2\sqrt{xy})\sum_{j=0}^{k-1}\frac{d_j}{(Lxy)^j}+
\cos(2\sqrt{xy})\sum_{j=0}^{k-1}\frac{e_j}{(Lxy)^{j+1/2}}
\right)+O((Lxy)^{-k-1/4}).
\end{multline}
Substituting \eqref{Y0+J0} to \eqref{phihat+ est2}, we prove that
\begin{equation}\label{phihat est2}
\widehat{\phi}\left(y\right)=\frac{-y^{1/4}}{\sqrt{\pi}}\sum_{\pm}
\int_0^{\infty}g_{\pm}(x)\exp\left(2ih_{\pm}(x)\right)dx+O\left(\frac{y^{1/4}}{(Ly)^k}\right),
\end{equation}
where
\begin{equation}\label{phihat h def}
h_{\pm}(x)=-TA(x)\pm\sqrt{Lxy},
\end{equation}
\begin{equation}\label{phihat g def}
g_{\pm}(x)=\frac{U(x)\exp\left(-G^2A^2(x)\right)}{x^{1/2}(x-(cq)^2/L)^{1/2}}V\left(\frac{xL-(cq)^2}{4cq},T\right)W_{\pm}(Lxy),
\end{equation}
and $W_{\pm}(Lxy)$ are finite sums  like in \eqref{Y0+J0}.
Applying \eqref{Ay deriv} we obtain
\begin{equation}\label{hpm deriv}
h'_{\pm}(x)=\frac{TcqL^{1/2}}{(Lx-(cq)^2)x^{1/2}}\pm\frac{\sqrt{Ly}}{2\sqrt{x}}.
\end{equation}
Using \eqref{e(G2A2) deriv} we infer
\begin{equation}\label{h+g+ jderiv}
h^{(j)}_{+}(x)\sim \frac{TQC}{\sqrt{L}}+\sqrt{Ly}\ll T^{\epsilon}\frac{TQC}{\sqrt{L}},\quad
g^{(j)}_{+}(x)\ll \left(GQC/\sqrt{L}+T^{\epsilon}\right)^j\ll T^{\epsilon j}.
\end{equation}
Then Lemma \ref{Lemma BKY}  with
\begin{equation}\label{BKY Lem8.1 conditions}
X=T^{\epsilon},\,R=\frac{TQC}{\sqrt{L}},\,Y=T^{\epsilon}\frac{TQC}{\sqrt{L}},\,P=1,\, V=T^{-\epsilon}
\end{equation}
yields the estimate
\begin{equation}\label{phihat est3}
\int_0^{\infty}g_{+}(x)\exp\left(2ih_{+}(x)\right)dx\ll T^{\epsilon}\left(\frac{T^{\epsilon}\sqrt{L}}{TCQ}\right)^{A}\ll T^{-A/2}.
\end{equation}
In the case of $h_{-}(x)$ there exists the saddle point $x_0$ such that
\begin{equation}\label{h- saddlepoint=}
Lx_0-(cq)^2=\frac{2Tcq}{\sqrt{y}}\quad\hbox{or}\quad
x_0=\frac{2Tcq}{L\sqrt{y}}+\frac{(cq)^2}{L}.
\end{equation}
If $\sqrt{y}\gg\frac{Tcq}{L}$ or $\sqrt{y}\ll\frac{Tcq}{L}$, then $x_0$ lies outside the interval of integration, and this means that on this interval $h'_{-}(x)$ is large,
namely $$h^{(j)}_{-}(x)\sim \frac{TQC}{\sqrt{L}}.$$ Therefore, for such $y$ we obtain an estimate similar to \eqref{phihat est3}.
In case of $\sqrt{y}\sim\frac{Tcq}{L}$ we apply Lemma \ref{Lemma Huxley} (analogously it is possible to use \cite[Proposition 8.2]{BKY}) with
\begin{equation}\label{HUX conditions}
\Theta_f=Tcq/\sqrt{L},\,
\Omega_f=1,\,
\Omega_g=T^{-\epsilon},
\end{equation}
proving that
\begin{equation}\label{phihat est4}
\int_0^{\infty}g_{-}(x)\exp\left(2ih_{-}(x)\right)dx=
g_{-}(x_0)\frac{\sqrt{\pi}\exp\left(2ih_{-}(x_0)+\pi i/4\right)}{\sqrt{h''_{-}(x_0)}}+O\left(T^{\epsilon}\frac{L^{3/4}}{(Tcq)^{3/2}}\right).
\end{equation}
Straightforward computations yield the equalities:
\begin{equation}\label{phihat h-2der}
h''_{-}(x_0)=\frac{-yL^{3/2}}{4Tcq\sqrt{x_0}}
\end{equation}
\begin{equation}\label{phihat h-(x0)}
h_{-}(x_0)=-T\arcosh\left(1+\frac{cq\sqrt{y}}{T}\right)-
\sqrt{2Tcq\sqrt{y}+y(cq)^2}.
\end{equation}
\begin{equation}\label{phihat g-(x0)}
g_{-}(x_0)=\frac{L^{1/2}y^{1/4}}{x_0^{1/2}(2Tcq)^{1/2}}U(x_0)\exp\left(-G^2A^2(x_0)\right)V\left(\frac{T}{2\sqrt{y}},T\right)W_{\pm}(Lx_0y).
\end{equation}
Substituting \eqref{phihat est4} to \eqref{phihat est2} we obtain
\begin{multline}\label{phihat est5}
\widehat{\phi}\left(y\right)=-\sqrt{2}\exp\left(2ih_{-}(x_0)-\pi i/4\right)\frac{U(x_0)\exp\left(-G^2A^2(x_0)\right)}{L^{1/4}x_0^{1/4}}
V\left(\frac{T}{2\sqrt{y}},T\right)W_{-}(Lx_0y)\\
+O\left(T^{\epsilon}\frac{L^{3/4}y^{1/4}}{(Tcq)^{3/2}}\right).
\end{multline}
Since we do not specify the function $W_{-}(Lx_0y)$, we put all constants in it. In the error term we may replace $y^{1/4}$ by $(Tcq/L)^{1/2}$.
Finally, we show that
\begin{multline}\label{phihat est6}
\widehat{\phi}\left(y\right)=\exp\left(2ih_{-}(x_0)\right)U\left(x_0\right)
\frac{\exp\left(-G^2\arcosh^2\left(1+\frac{cq\sqrt{y}}{T}\right)\right)}{L^{1/4}x_0^{1/4}}
V\left(\frac{T}{2\sqrt{y}},T\right)W_{-}(Lx_0y)\\
+O\left(\frac{L^{1/4}T^{\epsilon}}{Tcq}\right).
\end{multline}
\end{proof}
\section{Proof of Theorem \ref{KY result}}
We split  the sum over $c$ in \eqref{DSaa CQL} into three cases:
\begin{equation}
c\equiv0\pmod{4},\quad (c,2)=1,\quad c\equiv2\pmod{4},
\end{equation}
and for each case apply the corresponding Voronoi summation formula given by Lemma \ref{Thm Voronoi an c0mod4}, Lemma \ref{Thm Voronoi an codd} or Lemma \ref{Thm Voronoi an c2mod4}, respectively.
According to the estimate \eqref{phi+ est0}, all main terms in the Voronoi formula are negligible. Moreover, the contribution of $n<0$ and $n\gg T^{\epsilon}(TQC^2)^2/L$  is also negligible by Lemma \ref{lem phihat large n}. Combining \eqref{STM 0 def}, \eqref{STM 1 def}, \eqref{STM 4 def}, \eqref{STM 5 def} and \eqref{STM 45 to STM1} we obtain
\begin{equation}\label{DSaa Voronoi eq1}
\DS_{\ast,\ast}(C,Q,L)\ll \DSV_{\ast,\ast,0}(C,Q,L)+
\DSV_{\ast,\ast,1}(C,Q,L)+\DSV_{\ast,\ast,2}(C,Q,L)+T^{-A},
\end{equation}
where
\begin{equation}\label{DSVaa0 def}
\DSV_{\ast,\ast,0}(C,Q,L)=
\sum_{q}\frac{U(q/Q)}{q^{1/2}}
\sum_{c\equiv0\pmod{4}}\frac{U(c/C)}{c^{1/2}}
\sum_{0<n\ll T^{\epsilon}(TQC^2)^2/L}\frac{\Zag_{n}(1/2)}{n^{1/2}}
\widehat{\phi}\left(\frac{4\pi^2 n}{c^2}\right)\STM_0(c,n),
\end{equation}
\begin{equation}\label{DSVaa1 def}
\DSV_{\ast,\ast,1}(C,Q,L)=
\sum_{q}\frac{U(q/Q)}{q^{1/2}}
\sum_{(c,2)=1}\frac{U(c/C)}{c^{1/2}}
\sum_{0<n\ll T^{\epsilon}(TQC^2)^2/L}\frac{\Zag_{4n}(1/2)}{n^{1/2}}
\widehat{\phi}\left(\frac{\pi^2 n}{c^2}\right)\STM_1(c,n),
\end{equation}
\begin{equation}\label{DSVaa2 def}
\DSV_{\ast,\ast,2}(C,Q,L)=
\sum_{q}\frac{U(q/Q)}{q^{1/2}}
\sum_{c\equiv2\pmod{4}}\frac{U(c/C)}{c^{1/2}}
\sum_{0<n\ll T^{\epsilon}(TQC^2)^2/L}\frac{\Zag_{4n}(1/2)}{n^{1/2}}
\widehat{\phi}\left(\frac{4\pi^2 n}{c^2}\right)\STM_1(c/2,n).
\end{equation}
The analysis of all three sums requires following the same steps (mainly due to \eqref{STM 45 to STM1}, \eqref{STM1 series}, \eqref{STM0 series2}), therefore we estimate only $\DSV_{\ast,\ast,0}(C,Q,L)$.
Applying Lemma \ref{lem phihat} we infer
\begin{multline}\label{DSVaa0 eq1}
\DSV_{\ast,\ast,0}(C,Q,L)=
\sum_{q}\frac{U(q/Q)}{q^{1/2}}
\sum_{c\equiv0\pmod{4}}\frac{U(c/C)}{c^{1/2}}
\sum_{n}
U\left(\frac{nL^2}{(TQC^2)^2}\right)
\frac{\Zag_{n}(1/2)\STM_0(c,n)}{n^{1/2}}\\\times
\left(\frac{\exp\left(2ih(q,4\pi^2 n)\right)}{L^{1/4}}F(c,q;4\pi^2 n)
+O\left(\frac{L^{1/4}T^{\epsilon}}{TCQ}\right)\right).
\end{multline}
To handle the contribution of the error term, we estimate $\STM_0(c,n)$ trivially by $c$, apply \eqref{LZag 2mom estimate}, proving that
\begin{multline}\label{DSVaa0 eq2}
\DSV_{\ast,\ast,0}(C,Q,L)=
\sum_{q}\frac{U(q/Q)}{q^{1/2}}
\sum_{n}
U\left(\frac{nL^2}{(TQC^2)^2}\right)
\frac{\Zag_{n}(1/2)}{n^{1/2}}\frac{\exp\left(2ih(q,4\pi^2 n)\right)}{L^{1/4}}\\\times
\sum_{c\equiv0\pmod{4}}\frac{U(c/C)\STM_0(c,n)}{c^{1/2}}F(c,q;4\pi^2 n)
+O\left(\frac{T^{\epsilon}C^{5/2}Q^{1/2}}{L^{3/4}}\right).
\end{multline}
Let $\W(c):=U(c/C)F(c,q;4\pi^2 n)$. Then using the Mellin inversion formula and \eqref{STM0 series2} we show that
\begin{multline}\label{DSVaa0 c sum}
\sum_{c\equiv0\pmod{4}}\frac{U(c/C)\STM_0(c,n)}{c^{1/2}}F(c,q;4\pi^2 n)=\frac{1}{2\pi i}\int_{(1/2)}\tilde{\W}(z)
\frac{L^{\ast}(z,n)}{\zeta_2(2z)}\overline{r_2(1/4+\bar{z}/2,n)}dz,
\end{multline}
where $\tilde{\W}(z)=\int_0^{\infty}\W(x)x^{z-1}dx.$ Making the change of variables $x=Cy$ and integrating by parts, we derive the estimate
$$\tilde{\W}(z)\ll C^{\Re{z}}((1+|z|)T^{-\epsilon})^{-A}.$$ Therefore, we can truncate the integral in \eqref{DSVaa0 c sum} at $\Im{z}=T^{2\epsilon}$ introducing a negligibly small error term. Applying \eqref{Last 2mom estimate} and \eqref{LZag 2mom estimate} we prove
\begin{multline}\label{DSVaa0  nc sum}
\sum_{n}U\left(\frac{nL^2}{(TQC^2)^2}\right)
\frac{\Zag_{n}(1/2)}{n^{1/2}}\frac{\exp\left(2ih(q,4\pi^2 n)\right)}{L^{1/4}}
\sum_{c\equiv0\pmod{4}}\frac{\W(c)\STM_0(c,n)}{c^{1/2}}\\\ll
\frac{C^{1/2}T^{\epsilon}}{L^{1/4}}\int_{-T^{2\epsilon}}^{T^{2\epsilon}}
\sum_{n}U\left(\frac{nL^2}{(TQC^2)^2}\right)
\frac{|\Zag_{n}(1/2)|}{n^{1/2}}|L^{\ast}(1/2+it,n)|dz+T^{-A}\ll
\frac{T^{1+\epsilon}QC^{5/2}}{L^{5/4}}.
\end{multline}
Substituting \eqref{DSVaa0  nc sum} to \eqref{DSVaa0 eq2} and taking into account that $C,Q,L$ satisfy the conditions in  \eqref{DSaa dyadic},  we obtain
\begin{equation}\label{DSVaa0 eq3}
\DSV_{\ast,\ast,0}(C,Q,L)\ll\frac{T^{1+\epsilon}Q^{3/2}C^{5/2}}{L^{5/4}}+
\frac{T^{\epsilon}C^{5/2}Q^{1/2}}{L^{3/4}}\ll\frac{T^{1+\epsilon}Q^{3/2}C^{5/2}}{L^{5/4}},
\end{equation}
which results in the estimate
 $$\DS_{\ast,\ast}(C,Q,L)\ll T^{1+\epsilon}Q^{3/2}C^{5/2}L^{-5/4}.$$
Substitution of this estimate into \eqref{DSaa dyadic} yields
\begin{equation}\label{DSaa estimate}
\DS_{\ast,\ast}\ll\sum_{C\ll T^{1+\epsilon}/G^2}\sum_{T^{\al}/C\ll Q\ll T^{1+\epsilon}/CG^2}\sum_{(QC)^2G^{2-\epsilon}\ll L\ll QCT^{1+\epsilon}}
\frac{T^{1+\epsilon}Q^{3/2}C^{5/2}}{L^{5/4}}\ll
\frac{T^{1+\epsilon}}{G^{5/2}}.
\end{equation}
Finally, substituting \eqref{DSaa estimate} to \eqref{M2 to DS} we show that
\begin{equation}\label{M2 final estimate}
\M_2\ll T^{1+\epsilon}G\left(1+\frac{T^{1/2}}{G^{5/2}}\right)\ll T^{1+\epsilon}G
\end{equation}
provided that $G\gg T^{1/5}.$ This completes the proof of Theorem \ref{KY result}.

\begin{rema}
While treating  $\DS_{\ast,4}$, $\DS_{4,\ast}$, $\DS_{4,4}$ (see \eqref{M2 to DS}, \eqref{DS def}) we  obtain analogues of \eqref{DSVaa0 eq1}. In fact, the only difference is that in $\exp\left(2ih(q,4\pi^2 n)\right)F(c,q;4\pi^2 n)$ one should replace $q$ (see Remark \ref{rem phi cases}) by $q/(c,4)$,  $4q/(c,4)$ or $4q/(c,16)$, respectively. Since
$c\equiv0\pmod{4}$,  changes required in the first and in the second cases are trivial. In the third case, one should split the sum over $c$ into three: $4\|c$, $8\|c$ and
$c\equiv0\pmod{16}$.  This will only result in insignificant change in the  multiple $r_2(s,n)$  in \eqref{STM0 series}, because now instead of $k\ge 2$ in \cite[(3.4)]{PRR} we have $k=2$, $k=3$ and $k\ge4.$
\end{rema}


\end{document}